\documentclass[12pt]{amsart}
\usepackage{amssymb}
\usepackage{amsmath}
\usepackage{amsfonts}

\newtheorem{theorem}{Theorem}
\theoremstyle{plain}

\newtheorem{corollary}{Corollary}

\newtheorem{definition}{Definition}

\newtheorem{proposition}{Proposition}
\newtheorem{remark}{Remark}

\numberwithin{equation}{section}

\begin{document}
\title{Almost periodic generalized functions}
\author{Chikh Bouzar}
\address{Oran Essenia University, Algeria }
\email{ch.bouzar@gmail.com}
\author{Mohammed Taha Khalladi}
\address{University of Adrar, Algeria}
\email{ktaha2007@yahoo.fr }
\date{}

\begin{abstract}
The aim of this paper is to introduce and to study an algebra of almost
periodic generalized functions containing the classical Bohr almost periodic
functions as well as almost periodic Schwartz distributions.
\end{abstract}

\maketitle

\section{Introduction}

The theory of uniformly almost periodic functions was introduced
and studied by H. Bohr, since then, many authors contributed to
the development of this theory. There exist three equivalent
definitions of uniformly almost periodic functions, the first
definition of H. Bohr, S. Bochner's definition and the definition
based on the approximation property, see \cite{001}. S. Bochner's
definition is more suitable for extension to Schwartz
distributions. L. Schwartz in \cite{006} introduced the basic
elements of almost periodic distributions.

The new generalized functions of \cite{002}, \cite{002'}, give a solution to
the problem of multiplication of distributions, these generalized functions
are currently the subject of many scientific works, see \cite{003} and \cite%
{004}.

The aim of this work is to introduce and to study an algebra of almost
periodic generalized functions containing the classical Bohr almost periodic
functions as well almost periodic Schwartz distributions.

\section{Almost periodic functions and distributions}

We consider functions and distributions defined on the whole one dimensional
space $\mathbb{R}$. Recall $\mathcal{C}_{b}$ the space of bounded and
continuous complex valued functions on $\mathbb{R}$ endowed with the norm $%
\left\Vert ~\right\Vert _{\infty }$ of uniform convergence on $\mathbb{R}$, $%
\left( \mathcal{C}_{b},\left\Vert ~\right\Vert _{\infty }\right) $ is a
Banach algebra.

\begin{definition}
(S. Bochner) A complex valued function \textrm{\ }$f$ defined and continuous
on $\mathbb{R}$ is called almost periodic, if for any sequence of real
numbers $\left( h_{n}\right) _{n}$ one can extract a subsequence $\left(
h_{n_{k}}\right) _{k}$ such that $\left( f\left( .+h_{n_{k}}\right) \right)
_{k}$ converges in $\left( \mathcal{C}_{b},\left\Vert ~\right\Vert _{\infty
}\right) $. Denoted by $\mathcal{C}_{ap}$ the space of almost periodic
functions.
\end{definition}

To recall Schwartz almost periodic distributions, we need some function
spaces, see \cite{006}. Let $p\in \left[ 1,+\infty \right] ,$ the space
\begin{equation*}
\mathcal{D}_{L^{p}}:=\left\{ \varphi \in \mathcal{C}^{\infty }:\varphi
^{\left( j\right) }\in L^{p},\forall j\in \mathbb{Z}_{+}\right\}
\end{equation*}%
endowed with the topology defined by the countable family of norms

\begin{equation*}
\left\vert \varphi \right\vert _{k,p}:=\underset{j\leq k}{\sum }\left\Vert
\varphi ^{\left( j\right) }\right\Vert _{L^{p}},\text{ }k\in \mathbb{Z}_{+},
\end{equation*}%
is a differential Frechet subalgebra of $\mathcal{C}^{\infty }.$ The
topological dual of $\mathcal{D}_{L^{1}},$ denoted by $\mathcal{D}%
_{L^{\infty }}^{\prime },$ is called the space of bounded distributions.

Let $h\in \mathbb{R}$ and $T\in \mathcal{D}^{\prime },$ the translate of $T$
by $h,$ denoted by $\tau _{h}T\ ,$ is defined as :%
\begin{equation*}
\left\langle \tau _{h}T\ ,\varphi \right\rangle =\left\langle T\ ,\tau
_{-h}\varphi \right\rangle ,\varphi \in \mathcal{D}\text{,}
\end{equation*}%
where $\tau _{-h}\varphi \left( x\right) =\varphi \left( x+h\right) .$

The definition and characterizations of an almost periodic distribution are
summarized in the following results.

\begin{theorem}
\label{Theor01}For any bounded distribution $T$ $\in $ $\mathcal{D}%
_{L^{\infty }}^{\prime },$ the following statements are equivalent :

i) The set $\left\{ \tau _{h}T,h\in \mathbb{R}\right\} $\ is relatively
compact in $\mathcal{D}_{L^{\infty }}^{\prime }.$

ii) $T\ast \varphi \in \mathcal{C}_{ap},\forall $ $\varphi \in \mathcal{D}.$

iii) $\exists $ $\left( f_{j}\right) _{j\leq k}\subset \mathcal{C}_{ap},$ $T=%
\underset{j\leq k}{\sum }f_{j}^{\left( j\right) }.$

$T$ $\in $ $\mathcal{D}_{L^{\infty }}^{\prime }$ is said almost periodic if
it satisfies any (hence every) of the above conditions.
\end{theorem}

\begin{definition}
The space of almost periodic distributions is denoted by $\mathcal{B}%
_{ap}^{\prime }$.
\end{definition}

Let recall the space of regular almost periodic functions.

\begin{definition}
The space of almost periodic infinitely differentiable functions on $\mathbb{%
R}$ is defined and denoted by%
\begin{equation*}
\mathcal{B}_{ap}=\left\{ \varphi \in \mathcal{D}_{L^{\infty }}:\varphi
^{\left( j\right) }\in \mathcal{C}_{ap\text{ }},\forall j\in \mathbb{Z}%
_{+}\right\} .
\end{equation*}
\end{definition}

Some, easy to prove, properties of $\mathcal{B}_{ap}$\ are given
in the following assertions.

\begin{proposition}
\label{prpr001}We have

$i)$ $\mathcal{B}_{ap}$ is a closed differential subalgebra of $\mathcal{D}%
_{L^{\infty }}.$

$ii)$ If $T\in \mathcal{B}_{ap}^{\prime }$ and $\varphi \in \mathcal{B}_{ap%
\text{ }}$, then $\varphi T\in \mathcal{B}_{ap}^{\prime }.$

$iii)$ $\mathcal{B}_{ap\text{ }}\ast L^{1}\subset \mathcal{B}_{ap\text{ }}.$

$iv)$ $\mathcal{B}_{ap\text{ }}=\mathcal{D}_{L^{\infty }}\cap \mathcal{C}_{ap%
\text{ }}.$
\end{proposition}

As a consequence of $\left( \text{iv}\right) ,$ we have the following result.

\begin{corollary}
If $v\in \mathcal{D}_{L^{\infty }}$ and $v\ast \varphi \in \mathcal{C}%
_{ap},\forall \varphi \in \mathcal{D},$ then $v\in \mathcal{B}_{ap}.$
\end{corollary}

\begin{remark}
It is important to mention that $\mathcal{B}_{ap}\mathcal{\subsetneqq }%
\mathcal{C}^{\infty }\cap \mathcal{C}_{ap}.$
\end{remark}

\section{Almost periodic generalized functions}

Let $I=\left] 0,1\right] $ and%
\begin{equation*}
\mathcal{M}_{L^{\infty }}=\left\{ \left( u_{\varepsilon }\right)
_{\varepsilon }\in \left( \mathcal{D}_{L^{\infty }}\right) ^{I},\forall k\in
\mathbb{Z}_{+},\exists m\in \mathbb{Z}_{+},\text{ }\left\vert u_{\varepsilon
}\right\vert _{k,\infty }=O\left( \varepsilon ^{-m}\right) ,\text{ }%
\varepsilon \longrightarrow 0\right\}
\end{equation*}%
\begin{equation*}
\mathcal{N}_{L^{\infty }}=\left\{ \left( u_{\varepsilon }\right)
_{\varepsilon }\in \left( \mathcal{D}_{L^{\infty }}\right) ^{I},\forall k\in
\mathbb{Z}_{+},\forall m\in \mathbb{Z}_{+},\text{ }\left\vert u_{\varepsilon
}\right\vert _{k,\infty }=O\left( \varepsilon ^{m}\right) ,\text{ }%
\varepsilon \longrightarrow 0\right\}
\end{equation*}

\begin{definition}
The algebra of bounded generalized functions, denoted by $\mathcal{G}%
_{L^{\infty }},$ is defined by the quotient%
\begin{equation*}
\mathcal{G}_{L^{\infty }}=\frac{\mathcal{M}_{L^{\infty }}}{\mathcal{N}%
_{L^{\infty }}}
\end{equation*}
\end{definition}

Define%
\begin{equation}
\begin{array}{c}
\mathcal{M}_{ap}=\left\{ \left( u_{\varepsilon }\right) _{\varepsilon }\in
\left( \mathcal{B}_{ap}\right) ^{I},\forall k\in \mathbb{Z}_{+},\exists m\in
\mathbb{Z}_{+},\left\vert u_{\varepsilon }\right\vert _{k,\infty }=O\left(
\varepsilon ^{-m}\right) ,\varepsilon \longrightarrow 0\right\} \\
\mathcal{N}_{ap}=\left\{ \left( u_{\varepsilon }\right) _{\varepsilon }\in
\left( \mathcal{B}_{ap}\right) ^{I},\forall k\in \mathbb{Z}_{+},\forall m\in
\mathbb{Z}_{+},\left\vert u_{\varepsilon }\right\vert _{k,\infty }=O\left(
\varepsilon ^{m}\right) ,\varepsilon \longrightarrow 0\right\}%
\end{array}
\label{ES01}
\end{equation}

The properties of $\mathcal{M}_{ap}$ and $\mathcal{N}_{ap}$ are summarized
in the following proposition.

\begin{proposition}
\label{pimah1}$i)$ The space $\mathcal{M}_{ap}$ is a subalgebra of $\left(
\mathcal{B}_{ap}\right) ^{I}.$

$ii)$ the space $\mathcal{N}_{ap}$ is an ideal of $\mathcal{M}_{ap}.$
\end{proposition}

\begin{proof}
$i)$ It follows from the fact that $\mathcal{B}_{ap}$ is an differential
algebra.

$ii)$ Let $\left( u_{\varepsilon }\right) _{\varepsilon }\in \mathcal{N}%
_{ap} $ and $\left( v_{\varepsilon }\right) _{\varepsilon }\in \mathcal{M}%
_{ap},$ we have
\begin{equation*}
\forall k\in \mathbb{Z}_{+},\exists m^{\prime }\in \mathbb{Z}_{+},\exists
c_{1}>0,\exists \varepsilon _{0}\in I,\forall \varepsilon <\varepsilon
_{0},\left\vert v_{\varepsilon }\right\vert _{k,\infty }<c_{1}\varepsilon
^{-m^{\prime }}.
\end{equation*}%
Take $m\in \mathbb{Z}_{+},$ then for $m^{\prime \prime }=m+m^{\prime },$ $%
\exists c_{2}>0$ such that $\left\vert u_{\varepsilon }\right\vert
_{k,\infty }<c_{2}\varepsilon ^{m^{\prime \prime }}.$ Since the family of
the norms $\left\vert u_{\varepsilon }\right\vert _{k,\infty }$\ is
compatible with the algebraic structure of $\mathcal{D}_{L^{\infty }}$, then
$\forall k\in \mathbb{Z}_{+},$ $\exists c_{k}>0$ such that
\begin{equation*}
\left\vert u_{\varepsilon }v_{\varepsilon }\right\vert _{k,\infty }\leq c_{k}%
\text{ }\left\vert u_{\varepsilon }\right\vert _{k,\infty }\text{ }%
\left\vert v_{\varepsilon }\right\vert _{k,\infty },
\end{equation*}%
consequently%
\begin{equation*}
\left\vert u_{\varepsilon }v_{\varepsilon }\right\vert _{k,\infty
}<c_{k}c_{2}\varepsilon ^{m^{^{\prime \prime }}}c_{1}\varepsilon
^{-m^{\prime }}\leq C\varepsilon ^{m},\text{ where }C=c_{1}c_{2}c_{k}.
\end{equation*}%
Hence $\left( u_{\varepsilon }v_{\varepsilon }\right) _{\varepsilon }\in
\mathcal{N}_{ap}.$
\end{proof}

The following definition introduces the algebra of almost periodic
generalized functions.

\begin{definition}
The algebra of almost periodic generalized functions is the quotient algebra%
\begin{equation*}
\mathcal{G}_{ap}=\frac{\mathcal{M}_{ap}}{\mathcal{N}_{ap}}
\end{equation*}
\end{definition}

We have a characterization of elements of $\mathcal{G}_{ap}$ similar to the
result $\left( ii\right) $ of theorem $\left( \ref{Theor01}\right) $ for
almost periodic distributions$.$

\begin{theorem}
Let $u=\left[ \left( u_{\varepsilon }\right) _{\varepsilon }\right] \in
\mathcal{G}_{L^{\infty }},$ the following assertions are equivalent :

$i)$ $u$ is almost periodic.

$ii)$ $u_{\varepsilon }\ast \varphi \in \mathcal{B}_{ap},$ $\forall
\varepsilon \in I,\forall \varphi \in \mathcal{D}.$
\end{theorem}

\begin{proof}
$i)\Longrightarrow ii)$ If $u\in \mathcal{G}_{ap},$ so for every $%
\varepsilon \in I$ we have $u_{\varepsilon }\in \mathcal{B}_{ap},$ then $%
u_{\varepsilon }\ast \varphi \in \mathcal{B}_{ap},\forall \varepsilon \in
I,\forall \varphi \in \mathcal{D}.$

$ii)\Longrightarrow i)$ Let $\left( u_{\varepsilon }\right) _{\varepsilon
}\in \mathcal{M}_{L^{\infty }}$ and $u_{\varepsilon }\ast \varphi \in
\mathcal{B}_{ap},$ $\forall \varepsilon \in I,\forall \varphi \in \mathcal{D}%
,$ therefor $u_{\varepsilon }\in \mathcal{B}_{ap}$ follows from theorem $%
\left( \ref{Theor01}\right) $ $\left( ii\right) $; it suffices to show that%
\begin{equation*}
\forall k\in \mathbb{Z}_{+},\exists m\in \mathbb{Z}_{+},\left\vert
u_{\varepsilon }\right\vert _{k,\infty }=O\left( \varepsilon ^{-m}\right)
,\varepsilon \longrightarrow 0,
\end{equation*}%
which follows from the fact that $u\in \mathcal{G}_{L^{\infty }}.$
\end{proof}

\begin{remark}
The characterization $\left( ii\right) $ does not depend on representatives.
\end{remark}

\begin{definition}
Denote by $\Sigma $ the subset of functions $\rho \in \mathcal{S}$ satisfying%
\begin{equation*}
\int \rho \left( x\right) dx=1\text{ and }\int x^{k}\rho \left( x\right)
dx=0,\forall k=1,2,...
\end{equation*}
\end{definition}

Set $\rho _{\varepsilon }\left( .\right) =\frac{1}{\varepsilon }\rho \left(
\frac{.}{\varepsilon }\right) ,\varepsilon >0.$

\begin{proposition}
Let $\rho \in \Sigma ,$ the map%
\begin{equation*}
\begin{array}{cccc}
i_{ap}: & \mathcal{B}_{ap}^{\prime } & \longrightarrow & \mathcal{G}_{ap} \\
& u & \longrightarrow & \left( u\ast \rho _{\varepsilon }\right)
_{\varepsilon }+\mathcal{N}_{ap},%
\end{array}%
\end{equation*}

is a linear embedding which commutes with derivatives.
\end{proposition}

\begin{proof}
Let $u\in \mathcal{B}_{ap}^{\prime },$ by characterization of almost
periodic distributions we have $u=\underset{\beta \leq m}{\sum }f_{\beta
}^{\left( \beta \right) },$ where $f_{\beta }\in \mathcal{C}_{ap},$ so $%
\forall \alpha \in \mathbb{Z}$,%
\begin{equation*}
\left\vert \left( u^{\left( \alpha \right) }\ast \rho _{\varepsilon }\right)
\left( x\right) \right\vert \leq \underset{\beta \leq m}{\sum }\frac{1}{%
\varepsilon ^{\alpha +\beta }}\underset{\mathbb{R}}{\int }\left\vert
f_{\beta }\left( x-\varepsilon y\right) \rho ^{\left( \alpha +\beta \right)
}\left( y\right) \right\vert dy,
\end{equation*}%
consequently, there exists $c>0$ such that%
\begin{equation*}
\underset{x\in \mathbb{R}}{\sup }\left\vert \left( u^{\left( \alpha \right)
}\ast \rho _{\varepsilon }\right) \left( x\right) \right\vert \leq \underset{%
\beta \leq m}{\sum }\frac{1}{\varepsilon ^{\alpha +\beta }}\left\Vert
f_{\beta }\right\Vert _{L^{\infty }\left( \mathbb{R}\right) }\underset{%
\mathbb{R}}{\int }\left\vert \rho ^{\left( \alpha +\beta \right) }\left(
y\right) \right\vert dy\leq \frac{c}{\varepsilon ^{\alpha +m}},
\end{equation*}%
i.e.%
\begin{equation*}
\left\vert u\ast \rho _{\varepsilon }\right\vert _{m^{\prime },\infty }=%
\underset{\alpha \leq m^{\prime }}{\sum }\underset{x\in \mathbb{R}}{\sup }%
\left\vert \left( u^{\left( \alpha \right) }\ast \rho _{\varepsilon }\right)
\left( x\right) \right\vert \leq \frac{c^{\prime }}{\varepsilon
^{m+m^{\prime }}}\text{, }c^{\prime }=\underset{\alpha \leq m^{\prime }}{%
\sum }\frac{c}{\varepsilon ^{\alpha }},
\end{equation*}%
this shows that $\left( u\ast \rho _{\varepsilon }\right) _{\varepsilon }\in
\mathcal{M}_{ap}.$ Let $\left( u\ast \rho _{\varepsilon }\right)
_{\varepsilon }\in \mathcal{N}_{ap},$ then $\underset{\varepsilon
\longrightarrow 0}{\lim }u\ast \rho _{\varepsilon }=0$ in $\mathcal{D}%
_{L^{\infty }}^{\prime },$ but $\underset{\varepsilon \longrightarrow 0}{%
\lim }u\ast \rho _{\varepsilon }=u$ in $\mathcal{D}_{L^{\infty }}^{\prime },$
this shows that $i_{ap}$ is an embedding. Finally we note that $i_{ap}$ is
linear, this results from the fact that the convolution is linear and that $%
i_{ap}\left( w^{\left( j\right) }\right) =\left( w^{\left( j\right) }\ast
\rho _{\varepsilon }\right) _{\varepsilon }=\left( w\ast \rho _{\varepsilon
}\right) _{\varepsilon }^{\left( j\right) }=\left( i_{ap}\left( w\right)
\right) ^{\left( j\right) }.$
\end{proof}

The space $\mathcal{B}_{ap}\mathcal{\ }$is embedded into $\mathcal{G}_{ap}$
canonically, i.e.%
\begin{equation*}
\begin{array}{cccc}
\sigma _{ap}: & \mathcal{B}_{ap} & \longrightarrow & \mathcal{G}_{ap} \\
& f & \longrightarrow & \left[ \left( f\right) _{\varepsilon }\right]
=\left( f\right) _{\varepsilon }+\mathcal{N}_{ap}%
\end{array}%
\end{equation*}

There is two ways to embed $f\in \mathcal{B}_{ap}$ into $\mathcal{G}_{ap}$.
Actually we have the same result.

\begin{proposition}
The following diagram%
\begin{equation*}
\begin{array}{lll}
\mathcal{B}_{ap} & \longrightarrow & \mathcal{B}_{ap}^{\prime } \\
& \sigma _{ap}\searrow & \text{ }\downarrow i_{ap} \\
&  & \mathcal{G}_{ap}%
\end{array}%
\end{equation*}%
is commutative.
\end{proposition}

\begin{proof}
Let $f\in \mathcal{B}_{ap},$ we prove that $\left( f\ast \rho _{\varepsilon
}-f\right) _{\varepsilon }\in \mathcal{N}_{ap}.$ By Taylor's formula and the
fact that $\rho \in \Sigma ,$ we obtain
\begin{equation*}
\left\Vert f\ast \rho _{\varepsilon }-f\right\Vert _{L^{\infty }}\leq
\varepsilon ^{m}\underset{x\in \mathbb{R}}{\sup }\underset{\mathbb{R}}{\int }%
\left\vert \frac{\left( -y\right) ^{m}}{m!}f^{\left( m\right) }\left(
x-\theta \left( x\right) \varepsilon y\right) \rho \left( y\right)
dy\right\vert ,
\end{equation*}%
then $\exists C_{m}>0,$ such that%
\begin{equation*}
\left\Vert f\ast \rho _{\varepsilon }-f\right\Vert _{L^{\infty }}\leq
\varepsilon ^{m}C_{m}\left\Vert f^{\left( m\right) }\right\Vert _{L^{\infty
}}\left\Vert y^{m}\rho \right\Vert _{L^{1}}.
\end{equation*}%
The same result can be obtained for all the derivatives of $f.$ Hence $%
\left( f\ast \rho _{\varepsilon }-f\right) _{\varepsilon }\in \mathcal{N}%
_{ap}.$
\end{proof}

The Colombeau algebra of tempered generalized functions on $\mathbb{C}$ is
denoted $\mathcal{G}_{\mathcal{T}}\left( \mathbb{C}\right) ,$ for more
details on $\mathcal{G}_{\mathcal{T}}\left( \mathbb{C}\right) $ see \cite%
{002} or \cite{003}$.$

\begin{proposition}
Let $u\in \mathcal{G}_{ap}$ and $F\in \mathcal{G}_{\mathcal{T}}\left(
\mathbb{C}\right) $, then
\begin{equation*}
F\circ u=\left[ \left( F\circ u_{\varepsilon }\right) _{\varepsilon }\right]
\end{equation*}

is a well defined element of $\mathcal{G}_{ap}.$
\end{proposition}

\begin{proof}
It follows from the classical case of composition, in context of Colombeau
algebra, we have $F\circ u_{\varepsilon }\in \mathcal{B}_{ap}$\ in view of
the classical results of composition and convolution.
\end{proof}

We recall a characterization of integrable distributions.

\begin{definition}
A distribution $v\in \mathcal{D}^{\prime }$ is said an integrable
distribution, denoted $v\in \mathcal{D}_{L^{1}}^{\prime },$ if and only if $%
v=\underset{i\leq l}{\sum }f_{i}^{\left( i\right) },$ where $f_{i}\in L^{1}.$
\end{definition}

\begin{proposition}
If $u=\left[ \left( u_{\varepsilon }\right) _{\varepsilon }\right] \in
\mathcal{G}_{ap}$ and $v\in \mathcal{D}_{L^{1}}^{\prime },$ then the
convolution $u\ast v$ defined by%
\begin{equation*}
\left( u\ast v\right) \left( x\right) =\left( \underset{\mathbb{R}}{\int }%
u_{\varepsilon }\left( x-y\right) v\left( y\right) dy\right) _{\varepsilon }+%
\mathcal{N}\left[ \mathbb{C}\right]
\end{equation*}%
is a well defined almost periodic generalized function$.$
\end{proposition}

\begin{proof}
Let $\left( u_{\varepsilon }\right) _{\varepsilon }\in \mathcal{M}_{ap}$ be
a representative of $u,$ then%
\begin{equation*}
\forall k\in \mathbb{Z}_{+},\exists m\in \mathbb{Z}_{+},\exists C>0,\exists
\varepsilon _{0}\in I,\forall \varepsilon \leq \varepsilon _{0},\left\vert
u_{\varepsilon }\right\vert _{k,\infty }<C\varepsilon ^{-m},
\end{equation*}%
since $v\in \mathcal{D}_{L^{1}}^{\prime }$ then $v=\underset{i\leq l}{\sum }%
f_{i}^{\left( i\right) },$ where $f_{i}\in L^{1}.$ For each $\varepsilon \in
I,$ $u_{\varepsilon }\ast v$ is an almost periodic infinitely differentiable
function. By Young inequality there exists $C>0$ such that%
\begin{equation*}
\left\Vert \left( u_{\varepsilon }\ast v\right) ^{\left( j\right)
}\right\Vert _{L^{\infty }}\leq C\underset{i\leq l}{\sum }\left\Vert
f_{i}\right\Vert _{L^{1}}\left\Vert u_{\varepsilon }^{\left( i+j\right)
}\right\Vert _{L^{\infty }},
\end{equation*}%
consequently%
\begin{equation*}
\left\vert u_{\varepsilon }\ast v\right\vert _{k,\infty }=O\left(
\varepsilon ^{-m}\right) ,\varepsilon \longrightarrow 0\text{,}
\end{equation*}%
this shows that $\left( u_{\varepsilon }\ast v\right) _{\varepsilon }\in
\mathcal{M}_{ap}$. Suppose that $\left( w_{\varepsilon }\right)
_{\varepsilon }\in \mathcal{M}_{ap}$ is another representative of $u,$ then
there exists $C>0$ such that%
\begin{eqnarray*}
\left\Vert \left( u_{\varepsilon }\ast v-w_{\varepsilon }\ast v\right)
\right\Vert _{L^{\infty }} &\leq &\underset{i\leq l}{\sum }\underset{\mathbb{%
R}}{\sup }\underset{\mathbb{R}}{\int }\left\vert \left( u_{\varepsilon
}-w_{\varepsilon }\right) ^{\left( i\right) }\left( x-y\right) \right\vert
\left\vert f_{i}\left( y\right) \right\vert dy \\
&\leq &C\underset{i\leq l}{\sum }\left\Vert f_{i}\right\Vert
_{L^{1}}\left\Vert \left( u_{\varepsilon }-w_{\varepsilon }\right) ^{\left(
i\right) }\right\Vert _{L^{\infty }},
\end{eqnarray*}%
as $\left( u_{\varepsilon }-w_{\varepsilon }\right) _{\varepsilon }\in
\mathcal{N}_{ap},$ so $\forall m\in \mathbb{Z}_{+},$%
\begin{equation*}
\left\vert \left( u_{\varepsilon }\ast v-w_{\varepsilon }\ast v\right)
\left( x\right) \right\vert =O\left( \varepsilon ^{m}\right) ,\varepsilon
\longrightarrow 0\text{.}
\end{equation*}%
We obtain the same result for $\left( u_{\varepsilon }\ast v-w_{\varepsilon
}\ast v\right) _{\varepsilon }^{\left( j\right) }.$ Hence $\left(
u_{\varepsilon }\ast v-w_{\varepsilon }\ast v\right) _{\varepsilon }\in
\mathcal{N}_{ap}.$
\end{proof}

If $u=\left[ \left( u_{\varepsilon }\right) _{\varepsilon }\right] \in
\mathcal{G}_{ap}$, taking the integral of each element $u_{\varepsilon }$ on
a compact, we obtain an element of $\mathbb{C}^{I}.$

\begin{definition}
Let $u=\left[ \left( u_{\varepsilon }\right) _{\varepsilon }\right] \in
\mathcal{G}_{ap}$ and $x_{0}\in \mathbb{R}$, define the primitive of $u$ by%
\begin{equation*}
U\left( x\right) =\left( \underset{x_{0}}{\overset{x}{\int }}u_{\varepsilon
}\left( t\right) dt\right) _{\varepsilon }+\mathcal{N}\left[ \mathbb{C}%
\right] \text{.}
\end{equation*}
\end{definition}

We give a generalized version of the classical Bohl-Bohr theorem.

\begin{proposition}
The primitive of an almost periodic generalized function is almost periodic
if and only if it is bounded generalized function.
\end{proposition}

\begin{proof}
$\left( \Longrightarrow \right) :$ It follows from the fact that $\mathcal{G}%
_{ap}\subset \mathcal{G}_{L^{\infty }}.$ $\left( \Longleftarrow \right) :$
Let $u=\left[ \left( u_{\varepsilon }\right) _{\varepsilon }\right] \in
\mathcal{G}_{ap}$ and let $U=\left[ \left( U_{\varepsilon }\right)
_{\varepsilon }\right] $ be its primitive$,$ since for each $\varepsilon \in
I,$ $u_{\varepsilon }\in \mathcal{B}_{ap}$ and $U_{\varepsilon }=\underset{%
x_{0}}{\overset{x}{\int }}u_{\varepsilon }\left( t\right) dt\in \mathcal{D}%
_{L^{\infty }}$ then by the classical result of Bohl-Bohr and $\mathcal{B}%
_{ap},$ for every $\varepsilon \in I$ we have $\underset{x_{0}}{\overset{x}{%
\int }}u_{\varepsilon }\left( t\right) dt\in \mathcal{B}_{ap}.$ We shows
that $\left( U_{\varepsilon }\right) _{\varepsilon }\in \mathcal{M}_{ap},$
i.e.
\begin{equation*}
\forall k\in \mathbb{Z}_{+},\exists m\in \mathbb{Z}_{+},\left\vert
U_{\varepsilon }\right\vert _{k,\infty }=\underset{j\leq k}{\sum }\underset{%
x\in \mathbb{R}}{\sup }\left\vert U_{\varepsilon }^{\left( j\right) }\left(
x\right) \right\vert =O\left( \varepsilon ^{-m}\right) ,\varepsilon
\longrightarrow 0.
\end{equation*}%
If $j=0,$ since $\left( U_{\varepsilon }\right) _{\varepsilon }\in \mathcal{M%
}_{L^{\infty }}.$ By hypothesis, we have%
\begin{equation*}
\underset{x\in \mathbb{R}}{\sup }\left\vert U_{\varepsilon }\left( x\right)
\right\vert =O\left( \varepsilon ^{-m}\right) ,\varepsilon \longrightarrow 0,
\end{equation*}%
which shows that $\left( U_{\varepsilon }\right) _{\varepsilon }\in \mathcal{%
M}_{ap}.$ If $j\geq 1,$ we have $\left\vert U_{\varepsilon }^{\left(
j\right) }\left( x\right) \right\vert =\left\vert u_{\varepsilon }^{\left(
j-1\right) }\left( x\right) \right\vert ,$ which gives%
\begin{equation*}
\underset{j\leq k}{\sum }\underset{x\in \mathbb{R}}{\sup }\left\vert
U_{\varepsilon }^{\left( j\right) }\left( x\right) \right\vert \leq \underset%
{j\leq k}{\sum }\underset{x\in \mathbb{R}}{\sup }\left\vert u_{\varepsilon
}^{\left( j-1\right) }\left( x\right) \right\vert ,
\end{equation*}%
consequently%
\begin{equation*}
\left\vert U_{\varepsilon }\right\vert _{k,\infty }=O\left( \varepsilon
^{-m}\right) ,\varepsilon \longrightarrow 0,
\end{equation*}%
i.e. $\left( U_{\varepsilon }\right) _{\varepsilon }\in \mathcal{M}_{ap}.$
\end{proof}

As in the classical theory, we introduce the notion of mean value within the
algebra $\mathcal{G}_{ap}.$

\begin{definition}
Let $u\in \mathcal{G}_{ap}$, the generalized mean value of $u,$ denoted by $%
M_{g}\left( u\right) ,$ is defined by%
\begin{equation*}
M_{g}\left( u\right) =\left( \underset{X\longrightarrow +\infty }{\lim }%
\frac{1}{X}\underset{0}{\overset{X}{\int }}u_{\varepsilon }\left( x\right)
dx\right) _{\varepsilon }+\mathcal{N}\left[ \mathbb{C}\right] ,
\end{equation*}%
where $\left( u_{\varepsilon }\right) _{\varepsilon }$\ is a representative
of $u.$
\end{definition}

The definition of $M_{g}\left( u\right) $\ is correct and does not depend on
representatives.

\begin{proposition}
$i)$ If $u=\left[ \left( u_{\varepsilon }\right) _{\varepsilon }\right] \in
\mathcal{G}_{ap}$, then $M_{g}\left( u\right) \in \widetilde{\mathbb{C}}.$

$ii)$ If $\left( u_{\varepsilon }\right) _{\varepsilon }\in \mathcal{N}_{ap}$%
, then $M_{g}\left( u\right) =0$ in $\widetilde{\mathbb{C}}.$
\end{proposition}

\begin{proof}
$i)$ Let $\varepsilon \in I,$ we have%
\begin{equation*}
\left\vert \underset{X\rightarrow +\infty }{\lim }\frac{1}{X}\underset{0}{%
\overset{X}{\int }}u_{\varepsilon }\left( x\right) dx\right\vert \leq
\underset{x\in \mathbb{R}}{\sup }\left\vert u_{\varepsilon }\left( x\right)
\right\vert ,
\end{equation*}%
as $\left( u_{\varepsilon }\right) _{\varepsilon }\in \mathcal{M}_{ap},$ so $%
\exists m\in \mathbb{Z}_{+},\underset{x\in \mathbb{R}}{\sup }\left\vert
u_{\varepsilon }\left( x\right) \right\vert =O\left( \varepsilon
^{-m}\right) ,\varepsilon \longrightarrow 0,$ hence%
\begin{equation*}
\left( \underset{X\rightarrow +\infty }{\lim }\frac{1}{X}\underset{0}{%
\overset{X}{\int }}u_{\varepsilon }\left( x\right) dx\right) _{\varepsilon
}\in \mathcal{E}_{M}\left[ \mathbb{C}\right] .
\end{equation*}

$ii)$ If $\left( u_{\varepsilon }\right) _{\varepsilon }\in \mathcal{N}%
_{ap}, $ i.e.%
\begin{equation*}
\left\vert \underset{X\rightarrow +\infty }{\lim }\frac{1}{X}\underset{0}{%
\overset{X}{\int }}u_{\varepsilon }\left( x\right) dx\right\vert =O\left(
\varepsilon ^{m}\right) ,\varepsilon \longrightarrow 0,\forall m\in \mathbb{Z%
}_{+},
\end{equation*}%
then $M_{g}\left( u\right) =0$ in $\widetilde{\mathbb{C}}.$
\end{proof}

We have compatibility of the generalized mean value with that of a
distribution as stated in the following.

\begin{proposition}
If $T\in \mathcal{B}_{ap}^{\prime }$, then $M_{g}\left( i_{ap}\left(
T\right) \right) =M\left( T\right) $ in $\mathbb{C}.$
\end{proposition}

\begin{proof}
We have
\begin{equation*}
M_{g}\left( i_{ap}\left( T\right) \right) =\underset{X\longrightarrow
+\infty }{\lim }\frac{1}{X}\underset{0}{\overset{X}{\int }}i_{ap}\left(
T\right) \left( h\right) dh=\left( \underset{X\longrightarrow +\infty }{\lim
}\frac{1}{X}\underset{0}{\overset{X}{\int }}\left( T\ast \rho _{\varepsilon
}\right) \left( h\right) dh\right) _{\varepsilon },
\end{equation*}%
where $\rho \in \Sigma $ and $\rho _{\varepsilon }\left( .\right) =\frac{1}{%
\varepsilon }\rho \left( \frac{.}{\varepsilon }\right) .$ Let $\varphi \in
\mathcal{D}$ and $\underset{\mathbb{R}}{\int }\varphi \left( x\right) dx=1,$
then%
\begin{equation*}
M\left( T\right) =\underset{X\longrightarrow +\infty }{\lim }\frac{1}{X}%
\underset{0}{\overset{X}{\int }}\left( T\ast \varphi \right) \left( h\right)
dh.
\end{equation*}%
We have%
\begin{equation*}
M_{g}\left( i_{ap}\left( T\right) \right) -M\left( T\right) =M\left( T\ast
\left( \rho _{\varepsilon }-\varphi \right) \right) ,\forall \varepsilon \in
I.
\end{equation*}%
In view of formula $\left( VI.9.2\right) $ of \cite{006}$,$ we obtain%
\begin{equation*}
M_{g}\left( i_{ap}\left( T\right) \right) -M\left( T\right) =M\left(
T\right) \underset{\mathbb{R}}{\int }\left( \rho _{\varepsilon }\left(
x\right) -\varphi \left( x\right) \right) dx,
\end{equation*}%
as $\forall \varepsilon \in I,\underset{\mathbb{R}}{\int }\left( \rho
_{\varepsilon }\left( x\right) -\varphi \left( x\right) \right) dx=0,$ then $%
M_{g}\left( i_{ap}\left( T\right) \right) =M\left( T\right) $ in $\mathbb{C}%
. $
\end{proof}

\begin{remark}
We can introduce a new association within $\mathcal{G}_{ap}$ with the aid of
the generalized mean value $M_{g}.$
\end{remark}

\begin{definition}
A generalized trigonometric polynomial is a generalized function\textrm{\ }$%
\left[ \left( P_{\varepsilon }\right) _{\varepsilon }\right] ,$ where
\begin{equation*}
P_{\varepsilon }\left( x\right) =\underset{n=1}{\overset{l}{\sum }}%
c_{\varepsilon ,n}e^{i\lambda _{\varepsilon ,n}x}\text{ \ },\text{ \ }\left(
c_{\varepsilon ,n}\right) _{\varepsilon }\in \widetilde{\mathbb{C}}\text{ \
and }\left( \lambda _{\varepsilon ,n}\right) _{\varepsilon }\in \widetilde{%
\mathbb{R}}\text{, }n=1,...,l.
\end{equation*}
\end{definition}

\begin{proposition}
Every generalized trigonometric polynomial is an almost periodic generalized
function.
\end{proposition}

\begin{proof}
Let $P_{\varepsilon }\left( x\right) =\underset{n=1}{\overset{l}{\sum }}%
c_{\varepsilon ,n}e^{i\lambda _{\varepsilon ,n}x}$ where $\left(
c_{\varepsilon ,n}\right) _{\varepsilon }\in \widetilde{\mathbb{C}}$ \ and $%
\left( \lambda _{\varepsilon ,n}\right) _{\varepsilon }\in \widetilde{%
\mathbb{R}}$, $n=1,...,l,$ we have%
\begin{equation*}
\forall \varepsilon \in I,\text{ }\underset{n=1}{\overset{l}{\sum }}%
c_{\varepsilon ,n}e^{i\lambda _{\varepsilon ,n}x}\in \mathcal{B}_{ap},
\end{equation*}%
moreover $\exists m\in \mathbb{Z}_{+},\exists m^{\prime }\in \mathbb{Z}%
_{+},\left\vert \lambda _{\varepsilon ,n}\right\vert =O\left( \varepsilon
^{-m}\right) $ and $\left\vert c_{\varepsilon ,n}\right\vert =O\left(
\varepsilon ^{-m^{\prime }}\right) ,\varepsilon \longrightarrow 0.$
Consequently $\forall k\in \mathbb{Z}_{+},$ $\exists C>0$ such that%
\begin{eqnarray*}
\left\vert P_{\varepsilon }\right\vert _{k,\infty } &\leq &\underset{j\leq k}%
{\sum }\underset{n=1}{\overset{l}{\sum }}\left\vert \lambda _{\varepsilon
,n}\right\vert ^{j}\left\vert c_{\varepsilon ,n}\right\vert \\
&\leq &\left( \underset{j\leq k}{\sum }\left\vert \lambda _{\varepsilon
,1}\right\vert ^{j}+\left\vert \lambda _{\varepsilon ,2}\right\vert
^{j}+...+\left\vert \lambda _{\varepsilon ,l}\right\vert ^{j}\right)
c^{\prime }\varepsilon ^{-m^{\prime }} \\
&\leq &k\left( c\varepsilon ^{-m}+c^{2}\varepsilon
^{-2m}+...+c^{k}\varepsilon ^{-km}\right) c^{\prime }\varepsilon
^{-m^{\prime }} \\
&\leq &C\varepsilon ^{-m^{\prime \prime }},\text{ where }C=c^{\prime
}kl\left( c+c^{2}+...+c^{k}\right) ,\text{ }m^{\prime \prime }=km+m^{\prime
},
\end{eqnarray*}%
this show that $\left( P_{\varepsilon }\right) _{\varepsilon }\in \mathcal{M}%
_{ap}.$ In a similar way we show that if $\left( \lambda _{\varepsilon
,n}\right) _{\varepsilon }\in \mathcal{N}\left[ \mathbb{R}\right] $ and $%
\left( c_{\varepsilon ,n}\right) _{\varepsilon }\in \mathcal{N}\left[
\mathbb{C}\right] ,$ then $\left( P_{\varepsilon }\left( x\right) \right)
_{\varepsilon }\in \mathcal{N}_{ap}.$
\end{proof}

Let $u\in \mathcal{G}_{ap}$ and $\widetilde{\lambda }=\left[ \left( \lambda
_{\varepsilon }\right) _{\varepsilon }\right] \in \widetilde{\mathbb{R}}$,
then $ue^{-i\widetilde{\lambda }x}=\left( u_{\varepsilon }e^{-i\lambda
_{\varepsilon }x}\right) _{\varepsilon }\in \mathcal{G}_{ap},$ so the
generalized mean value $M_{g}\left( ue^{-i\widetilde{\lambda }x}\right) $ is
a well defined element of $\widetilde{\mathbb{C}}$. Define%
\begin{equation*}
a_{\widetilde{\lambda }}\left( u\right) =M_{g}\left( ue^{-i\widetilde{%
\lambda }x}\right) ,
\end{equation*}%
and the generalized spectra of $u,$%
\begin{equation*}
\Lambda _{g}\left( u\right) =\left\{ \widetilde{\lambda }\in \widetilde{%
\mathbb{R}}:a_{\widetilde{\lambda }}\left( u\right) \neq 0\text{ in }%
\widetilde{\mathbb{C}}\right\} .
\end{equation*}

\begin{remark}
If a function or a distribution $u$ is almost periodic, then its spectra is
at most countable, see \cite{001}\ and \cite{006}$.$
\end{remark}

\begin{proposition}
\label{prpr002}Let $P=\left[ \left( \underset{n=1}{\overset{l}{\sum }}%
c_{\varepsilon ,n}e^{i\lambda _{\varepsilon ,n}x}\right) _{\varepsilon }%
\right] $ be a generalized trigonometric polynomial, then%
\begin{equation*}
\Lambda _{g}\left( P\right) =\left\{ \left[ \left( \lambda _{\varepsilon
,n}\right) _{\varepsilon }\right] :n=1,...,l\right\} .
\end{equation*}
\end{proposition}

\begin{proof}
A direct computation gives the result.
\end{proof}

\begin{remark}
We are indebted to the referee for the following example of an almost
periodic generalized function having uncountable generalized spectra. Let $u=%
\left[ \left( u_{\varepsilon }\right) _{\varepsilon }\right] ,$ $%
u_{\varepsilon }\left( x\right) =1+e^{ix},\forall \varepsilon \in I,$ by
proposition $\left( \ref{prpr002}\right) ,$ $\Lambda _{g}\left( u\right)
=\left\{ \left[ \left( A_{\varepsilon }\right) _{\varepsilon }\right]
:A_{\varepsilon }=\left\{ 0,1\right\} \right\} ,$ as the set $\left\{ \left(
\lambda _{\varepsilon }\right) _{\varepsilon }:\lambda _{\varepsilon
}=\left\{ 0,1\right\} ,\varepsilon \in I\right\} \subset \Lambda _{g}\left(
u\right) $ is uncountable, then the generalized spectra of $u$ is too.
\end{remark}

\end{document}